\documentclass{amsart}
\usepackage{amssymb,amsmath,amscd,amsthm}
\usepackage{amsfonts,latexsym}

\date{\today}

\newcommand{\Z}{{\mathbb Z}}
\newcommand{\R}{{\mathbb R}}

\newcommand{\N}{{\mathbb N}}

\newcommand{\D}{{\mathbb D}}

\newcommand{\A}{{\mathcal A}}

\newtheorem{theorem}{Theorem}[section]
\newtheorem{lemma}[theorem]{Lemma}
\newtheorem{prop}[theorem]{Proposition}
\newtheorem{coro}[theorem]{Corollary}
\newtheorem{conj}[theorem]{Conjecture}
\theoremstyle{definition}

\theoremstyle{definition}
\newtheorem{defi}[theorem]{Definition}
\newtheorem{exam}[theorem]{Example}

\begin{document}

\title[Schr\"odinger Operators With Pattern Sturmian Potentials]{Spectral Properties of Schr\"odinger Operators With Pattern Sturmian Potentials}

\author[D.\ Damanik]{David Damanik}

\address{Department of Mathematics, Rice University, Houston, TX~77005, USA}

\email{damanik@rice.edu}

\thanks{D.\ D.\ was supported in part by NSF grants DMS--1067988 and DMS--1361625.}

\author[Q.-H.\ Liu]{Qing-Hui Liu}

\address{Department of Computer Science, Beijing Institute of Technology, Beijing 100081, P.R.\ China}

\email{qhliu@bit.edu.cn}

\thanks{Q.-H.\ L.\ was supported by the National Natural Science Foundation of China, No. 11371055.}

\author[Y.-H.\ Qu]{Yan-Hui Qu}

\address{Department of Mathematics, Tsinghua University, Beijing 100084, P.R.\ China}

\email{yhqu@math.tsinghua.edu.cn}

\thanks{Y.-H.\ Q.\ was supported by the National Natural Science Foundation of China, No. 11201256, 11371055 and 11431007.}

\begin{abstract}
We consider discrete Schr\"odinger operators with pattern Sturmian potentials. This class of potentials strictly contains the class of Sturmian potentials, for which the spectral properties of the associated Schr\"odinger operators are well understood. In particular, it is known that for every Sturmian potential, the associated Schr\"odinger operator has zero-measure spectrum and purely singular continuous spectral measures. We conjecture that the same statements hold in the more general class of pattern Sturmian potentials. We prove partial results in support of this conjecture. In particular, we confirm the conjecture for all pattern Sturmian potentials that belong to the family of Toeplitz sequences.
\end{abstract}

\maketitle

\section{Introduction}\label{s.intro}

The classical Morse-Hedlund theorem states that a two-sided sequence over a finite alphabet is periodic if and only if there exists $n \in \Z_+$ such that $p(n) \le n$. Here $p(n)$ denotes the numbers of subwords of length $n$ of the given sequence. Put differently, the sequence is aperiodic if and only if $p(n) \ge n+1$ for every $n$. Recurrent aperiodic sequences of minimal complexity (i.e., $p(n) = n+1$ for every $n$) are called Sturmian sequences. These sequences play an important role in the study of one-dimensional quasicrystals for obvious reasons (aperiodicity coupled with strong order properties).

From an electronic transport perspective in one-dimensional quasicrystals, it is therefore of interest to study discrete one-dimensional Schr\"odinger operators
\begin{equation}\label{e.oper}
[H \psi](n) = \psi(n+1) + \psi(n-1) + V(n) \psi(n)
\end{equation}
in $\ell^2(\Z)$ with a Sturmian $V$. In order to ensure the self-adjointness of the operators we consider, we will assume throughout this paper that whenever a symbolic sequence serves as the potential of a Schr\"odinger operator, the underlying alphabet is a subset of $\R$.

Our study is motivated by the following theorem.

\begin{theorem}\label{t.sturmian}
Suppose $H$ is a discrete Schr\"odinger operator with a Sturmian potential. Then the spectrum of $H$ has zero Lebesgue measure and its spectral measures are purely singular continuous.
\end{theorem}

The zero-measure property (which in turn also implies the absence of absolutely continuous spectrum) was shown by Bellissard et al.\ \cite{BIST89} and the absence of eigenvalues (which, together with the previous result, implies purely singular continuous spectrum) was shown by Damanik et al.\ in \cite{DKL00}.

In recent years there has been extensive work on a generalization of Sturmian sequences. Using maximal pattern complexity instead of standard combinatorial complexity, a similar characterization of recurrent aperiodic sequences has been found and those recurrent aperiodic sequences of minimal maximal pattern complexity are called pattern Sturmian sequences; compare, for example, \cite{GKTX06, KRX06, KZ02, KZ02b} and references therein.

The maximal pattern complexity function $p^* = p^*_V$ of a given recurrent two-sided sequence $V$ is defined as follows,
$$
p^*(n) = \sup_\tau \# \{ V(m + \tau(0)) V(m + \tau (1)) \cdots V(m + \tau(n-1) : m \in \Z \},
$$
where the $\sup$ is taken over all $(\tau(0), \tau(1), \ldots, \tau(n-1))$ with $0 = \tau(0) < \tau(1) < \cdots < \tau(n-1)$ and $\tau(j) \in \Z_+$ for $j = 1, \ldots, n-1$. A recurrent two-sided sequence is aperiodic if and only if $p^*(n) \ge 2n$ for every $n$.\footnote{A version of this statement is shown for one-sided sequences in \cite{KZ02}. Namely, a one-sided sequence is eventually periodic if and only if $p^*(n) \le 2n-1$ for some $n$. It is easy to modify the proof to obtain the statement above for recurrent two-sided sequences. Compare, for example, \cite[Proposition~1.2]{KRX06}.} Recurrent aperiodic sequences of minimal maximal pattern complexity (i.e., $p^*(n) = 2n$ for every $n$) are called pattern Sturmian sequences.

It turns out that all Sturmian sequences are pattern Sturmian, but there are many more examples as we will recall below. Thus, it is natural to ask whether the following generalization of Theorem~\ref{t.sturmian} holds.

\begin{conj}\label{c.patternsturmian}
Suppose $H$ is a discrete Schr\"odinger operator in $\ell^2(\Z)$ with a pattern Sturmian potential. Then the spectrum of $H$ has zero Lebesgue measure and its spectral measures are purely singular continuous.
\end{conj}

Since the theory of pattern Sturmian sequences is usually discussed in the setting of one-sided sequences, let us consider this case now. The maximal pattern complexity function of a one-sided sequences $V$ (defined on $\Z_+$) is given by
$$
p^*(n) = \sup_\tau \# \{ V(m + \tau(0)) V(m + \tau (1)) \cdots V(m + \tau(n-1) : m \in \Z_+ \}
$$
and the fundamental result is now that $V$ is eventually periodic if and only if $p^*(n) \le 2n-1$ for some $n$; see \cite{KZ02}. A one-sided sequence with $p^*(n) = 2n$ for every $n \in \Z_+$ is called pattern Sturmian. It is important to note that not all pattern Sturmian sequences are recurrent (while it is well known that one-sided sequences of minimal complexity $p(n) = n+1$ are recurrent, as they are given precisely by the half-line restrictions of the two-sided Sturmian sequences).

The Schr\"odinger operator in $\ell^2(\Z_+)$ associated with a one-sided sequence $V$ acts as in \eqref{e.oper} for $n \in \Z_+$, and one imposes a boundary condition of the form $\psi(0) \sin \phi + \psi(1) \cos \phi = 0$, which in turn makes the operator self-adjoint. Sometimes one makes the dependence of the operator on the boundary condition explicit by writing $H_\phi$ instead of $H$.

Unlike in the classical case, there is as yet no complete classification of all pattern Sturmian sequences. Instead, the existing papers have focused on identifying suitable subclasses of pattern Sturmian sequences.

\begin{exam}\label{example}
Here are the known classes of pattern Sturmian sequences:

\begin{enumerate}

\item \textit{Circle map sequences}: A circle map sequence is one of the form $V(n) = \lambda \chi_{[1-\beta, 1)}(n \alpha + \theta \!\!\! \mod 1)$, where $\lambda \not= 0$, $\alpha \in (0,1)$ is irrational, $\beta \in (0,1)$, and $\theta \in [0,1)$. Such a sequences is Sturmian if and only if $\beta = \alpha$. Circle map sequences were shown to be pattern Sturmian by Kamae and Zamboni in \cite{KZ02}. All circle map sequences are recurrent and can be considered in both the one-sided and two-sided settings.

\item \textit{Toeplitz sequences}: Suppose $\mathcal{A}$ is a finite alphabet and $?$ is a letter that does not belong to $\mathcal{A}$. Denote by $\mathcal{P}(\mathcal{A}, ?)$ the set of periodic sequences $\xi = \eta^\infty$ (one-sided or two-sided, depending on the setting) for which every symbol in $\mathcal{A}$ occurs as least once in $\eta$, while $?$ occurs exactly once in $\eta$. If $\xi^1 \in \mathcal{P}(\mathcal{A}, ?)$ and $\xi^2 \in \mathcal{P}(\mathcal{B}, ?)$, $\xi^1 \lhd \xi^2$ is obtained by substituting $?$ by $\ldots, \xi^2(j-1), \xi^2(j) \xi^2(j+1), \ldots$, in that order. Then $\xi^1 \lhd \xi^2 \in \mathcal{P}(\mathcal{A} \cup \mathcal{B}, ?)$. Iterating this procedure indefinitely we obtain a so-called Toeplitz sequence over the alphabet given by the union of the alphabets used in the process, assuming it is finite. The words $\xi^1, \xi^2, \ldots$ used in the iterative procedure are called the coding words of the given Toeplitz sequence. All Toeplitz sequences are recurrent and can be considered in both the one-sided and two-sided settings. A Toeplitz sequence is called simple if it is defined over a $2$-letter alphabet, say $\mathcal{A} = \{ a, b \}$, and it has a sequence of coding words $\xi^1, \xi^2, \ldots  \in \mathcal{P}(\{a\}, ?) \cup \mathcal{P}(\{b\}, ?)$. Gjini et al.\ showed in \cite{GKTX06} that every simple Toeplitz sequence is pattern Sturmian. More precisely, they gave a characterization of all Toeplitz sequences that are pattern Sturmian; see \cite{GKTX06} for details. For our purposes it is sufficient to note that the results of \cite{GKTX06} imply that every pattern Sturmian Toeplitz sequence must be of the form $\eta^\infty \lhd \xi$, where $\eta$ is a finite word containing exactly one $?$ and $\xi$ is a simple Toeplitz sequence.

\item \textit{Sparse sequences}: Consider the one-sided sparse sequence
$$
V(n) = \begin{cases} v & n = n_k \text{ for some } k, \\ 0 & \text{otherwise}, \end{cases}
$$
where $n_{k+1} > 2 n_k \ge 1$. This sequence is pattern Sturmian; compare \cite[Example~5]{KZ02}. Sparse sequences are not recurrent and are naturally considered in the one-sided setting.

\end{enumerate}

It would certainly be of interest to find additional examples or to establish limitations on existing ones beyond the known ones (compare, e.g., \cite{KRTX06}).
\end{exam}

The third class of examples shows that we cannot expect a statement like Conjecture~\ref{c.patternsturmian} to hold for one-sided pattern Sturmian potentials. Indeed for the sparse potentials from that class of examples, it is easy to see that the spectrum of the associated Schr\"odinger operator will contain the interval $[-2,2]$ and hence will not have zero Lebesgue measure. Moreover it is also easy to see (and will be explained in more detail below) that the spectrum outside $[-2,2]$ is non-empty and pure point, and hence the operators will not have purely singular continuous spectrum. We therefore can only formulate the following modest conjecture in the one-sided case.

\begin{conj}\label{c.patternsturmian1sided}
Suppose $H_\phi$ is a discrete Schr\"odinger operator in $\ell^2(\Z_+)$ with a pattern Sturmian potential and boundary condition $\phi$. Then the absolutely continuous spectrum of $H$ is empty and its singular continuous spectrum is non-empty.
\end{conj}

Again, as the examples above show, the spectrum may or may not be of zero Lebesgue measure, and additional point spectrum may be present.

Given that the class of all pattern Sturmian sequences is not yet fully understood (and the classification of all Sturmian sequences is crucial to the proof of Theorem~\ref{t.sturmian}), it is more or less hopeless to attack Conjectures~\ref{c.patternsturmian} and \ref{c.patternsturmian1sided} head-on. Rather we will prove the claims in these conjectures for as many of the known examples of pattern Sturmian sequences as possible.

Specifically we will devote one section to each of the three classes of examples of pattern Sturmian sequences mentioned above. We discuss circle map sequences in Section~\ref{s.circlemap}, pattern Sturmian Toeplitz sequences in Section~\ref{s.toeplitz}, and sparse pattern Sturmian sequences in Section~\ref{s.sparse}.

\section{Circle Map Sequences}\label{s.circlemap}

In this section we discuss circle map sequences, the first class mentioned in Example~\ref{example}. Within this class we are unfortunately still far away from fully establishing Conjecture~\ref{c.patternsturmian}. We begin by summarizing the known partial results and then proceed to discuss how additional progress may be made.

The following result on zero-measure spectrum was established in \cite{DL06b}.

\begin{theorem}\label{t.cmzms}
For a discrete Schr\"odinger operator $H$ in $\ell^2(\Z)$ with a potential given by a circle map sequence corresponding to the parameters parameters $\lambda \not= 0$, $\alpha$ irrational, $\beta \in (0,1)$, and $\theta$, zero-measure spectrum holds in the following range of parameters:
\begin{itemize}

\item every $\lambda \not= 0$, every irrational $\alpha$, every $\beta \in \left( \Z \alpha + \Z \right) \cap (0,1)$, every $\theta$;

\item every $\lambda \not= 0$, every irrational $\alpha$, Lebesgue almost every $\beta \in (0,1)$, every $\theta$;

\item every $\lambda \not= 0$, every irrational $\alpha$ that is of bounded type, every $\beta \in (0,1)$, every $\theta$.

\end{itemize}
\end{theorem}

Of course $\lambda \not= 0$ and $\alpha$ irrational are necessary conditions for zero-measure spectrum (as the potential is periodic otherwise and hence the spectrum is given by a finite union of non-degenerate compact intervals), and the spectrum in this case is actually independent of $\theta$. Thus, the remaining question is whether we can improve for given irrational $\alpha$ the full-measure condition on $\beta$ to all possible $\beta$'s. It is known that the method (developed in \cite{DL06}) used to prove zero-measure spectrum in the parameter range described above will fail, for any given irrational $\alpha$ that is not of bounded type, for some $\beta$; see \cite{DL06b}. Thus, new ideas are required to establish zero-measure spectrum for all circle map sequences.

Generally speaking, there are two fundamental ways of proving zero-measure spectrum as a consequence of Kotani's seminal work \cite{K89}: the identification of $S$-adic structures in the potential and a trace-map analysis based on these structures, and the verification of the so-called Boshernitzan condition, which in turn implies the absence of non-uniform hyperbolicity for the associated Schr\"odinger cocycles. The discussion in the previous paragraph is implicitly centered around the approach using the Boshernitzan condition. It has been verified in the generality described in Theorem~\ref{t.cmzms}, and it is known to not always hold as mentioned above. Thus, one either has to come up with a completely new way of proving zero-measure spectrum, or one has to somehow make the approach based on $S$-adic structures and trace maps work.

Let us discuss how the latter way might be carried out. It is well known that circle map sequences have complexity $p(n) = 2n$ when $\beta \not\in \Z + \Z \alpha$ (otherwise they are quasi-Sturmian and the complexity is $p(n) = n+k$); compare \cite{R94}. Moreover, a linearly bounded complexity function ensures that weak $S$-adic structures are present \cite{F96}. That is, there is a finite set of substitutions such that suitable iterations of these substitutions generate all the subwords of the given sequence. Any substitution generates an associated trace map. Thus, one could in principle hope that an analysis such as the one pioneered in \cite{S87, S89} and developed further, for example, in \cite{BIST89, BG93, LTWW02} can be carried out. While this may seem like a daunting task, we want to point out that the results obtained in \cite{J91} strongly hint at the fact that this undertaking is manageable and may produce the desired results.

Let us turn to the spectral type.

\begin{theorem}
Consider a discrete Schr\"odinger operator $H$ in $\ell^2(\Z)$ with a potential given by a circle map sequence corresponding to the parameters $\lambda \not= 0$, $\alpha$ irrational, $\beta \in (0,1)$, and $\theta$. Then, the absolutely continuous spectrum of $H$ is empty. Absence of point spectrum holds in the following range of parameters:
\begin{itemize}

\item every $\lambda \not= 0$, every $\alpha$ irrational, $\beta \in \left( \Z \alpha + \Z \right) \cap (0,1)$, every $\theta$;

\item every $\lambda \not= 0$, every $\beta \in (0,1)$, Lebesgue almost every $\alpha$, Lebesgue almost every $\theta$.

\end{itemize}
\end{theorem}

The absence of absolutely continuous spectrum follows from \cite{K89}, combined with \cite{K97} or \cite{LS99}, or alternatively from \cite{R11}. The first result on the absence of point spectrum was shown in \cite{DL03} and the second was shown in \cite{DP86}.

As we see, the absence of absolutely continuous spectrum is known in full generality. Thus, the issue is to exclude eigenvalues also in full generality. There are three fundamental ways of excluding eigenvalues: using palindromes \cite{HKS95} , squares (along with trace estimates) \cite{S87}, or cubes \cite{DP86}. It was shown in \cite{DZ00} that in many cases palindromes alone can never yield results for all $\theta$, while it was shown in \cite{D00} that cubes alone can never yield results for all $\theta$. Thus, one is more or less forced to include the method based on squares and trace estimates in one's considerations, and this is indeed how all the known results that hold for all $\theta$ have been obtained.

This in turn makes a connection to what was said above. The required trace estimates are usually obtained via a trace map analysis. Thus, for both the zero-measure property and the spectral type, to fully establish Conjecture~\ref{c.patternsturmian} for potentials given by circle map sequences, it seems it will be crucial to make the approach based on $S$-adic structures and trace map analysis work.

\section{Pattern Sturmian Toeplitz Sequences}\label{s.toeplitz}

In this section we discuss pattern Sturmian Toeplitz sequences, the second class mentioned in Example~\ref{example}. Within this class we are able to fully establish Conjecture~\ref{c.patternsturmian}:

\begin{theorem}\label{t.patternsturmian}
Suppose $H$ is a discrete Schr\"odinger operator in $\ell^2(\Z)$ with a potential given by a pattern Sturmian Toeplitz sequence. Then the spectrum of $H$ has zero Lebesgue measure and its spectral measures are purely singular continuous.
\end{theorem}

The result is known in some cases. Liu and Qu studied simple Toeplitz sequences in \cite{LQ11} and established zero-measure spectrum for all of them and purely singular continuous spectrum for many of them. Here we take their analysis further and prove the two statements for the entire class of pattern Sturmian Toeplitz sequences.

The well-known period doubling substitution sequence (or rather each of the elements of the subshift it generates) is a simple Toeplitz sequence. For this model, purely singular continuous spectrum for all elements of the subshift was shown in \cite{D01}. Our proof of singular continuous spectrum for all pattern Sturmian Toeplitz sequences is inspired by this proof.

A central component of the proof is the study of the traces of the transfer matrices over the basic building blocks of the given Toeplitz sequence. We begin by proving a characterization of the energies in the spectrum in terms of the behavior of these traces; see Proposition~\ref{p.tracesandspectrum} below.

\subsection{Uniform Convergence of Cocycles}

$(X,T) $ is called a {\it topological dynamical system} ({\it TDS}) if $X$ is a compact metric space and $T : X \to X$ is a homeomorphism. A TDS $(X,T)$ is called {\it minimal} if every orbit is dense; it is called {\it uniquely ergodic} if there is only one $T$-invariant probability measure on $X$. It is called {\it strictly ergodic} if it is both minimal and uniquely ergodic.

Let ${\rm SL}(2,\R)$ be the group of real valued $2 \times 2$ matrices with determinant equal to one, equipped with the topology induced by the standard matrix norm on square matrices.

For a continuous function $A : X \to {\rm SL}(2,\R)$, $x \in X$, and
$n\in\Z$, we define the cocycle $A(n,x)$ by
$$
A(n,x)=
\begin{cases}
A(T^{n-1}x)\cdots A(x); & n>0\\
{\rm Id};& n=0\\
A^{-1}(T^{n}x)\cdots A^{-1}(T^{-1}x); & n<0.
\end{cases}
$$
By Kingman's subadditive ergodic theorem, there exists $\Lambda(A)\in\R$ with
\begin{equation}\label{subadditive}
\Lambda(A) = \lim_{n \to \infty} \frac{1}{n} \ln \|A(n,x)\|
\end{equation}
for $\mu$ a.e.\ $x \in X$ if $(X,T,\mu)$ is an ergodic dynamical system. It is well known that the unique ergodicity of $(X,T)$ is equivalent to uniform convergence in the Birkhoff  additive ergodic theorem when applied to continuous functions. Motivated by this, Furman \cite{F} gives the following definition for cocycles.

\begin{defi}(\cite{F})
Let $(X,T)$ be strictly ergodic. A continuous function $A : X \to {\rm GL}(2,\R) $ is called \textit{uniform} if the limit
$$
\displaystyle \Lambda(A) = \lim_{n \to \infty} \frac{1}{n} \ln \|A(n,x)\|
$$
exists for all $x \in X$ and the convergence is uniform on $X$.
\end{defi}

\subsection{Toeplitz Subshifts}\label{toeplitz}

Toeplitz words have been extensively studied since the work \cite{JK}. These words are constructed by starting with the ``empty sequence''
and successively filling the ``holes'' with periodic sequences. In this paper we will consider a more general family of Toeplitz words than those studied in \cite{KZ02, LQ11, QRWX10}.

Given a finite alphabet $\A$, let $\A^n$ be the set of finite words over $\A$ of length $n$. For $w \in \A^n$, we denote the length of $w$
by $|w|=n$. Let $\displaystyle \A^\ast := \bigcup_{n \geq 0} \A^n$. For two words $u, v \in \A^\ast$, if there exist $v^1, v^2 \in \A^\ast$ such
that $u = v^1 v v^2$, then we say that $v$ is a factor of $u$ and write $v \prec u$. Let ${\rm FG}(\A)$ be the free group generated by $\A$. Throughout this paper, $a^{-1}$ always denotes the inverse of $a$, where $a$ is viewed as an element of ${\rm FG}(\A)$.

Now we begin to define the concept of a Toeplitz word. The definition we give here is essentially taken from \cite{QRWX10}.

\textbf{Partial word.} Let ${\mathcal A} \subset \R$ be a finite set with $\#{\mathcal A} \geq 2$. Let $\alpha$ be a word in $({\mathcal A} \cup \{?\})^\Z$. Let us denote by $U = \{ x \in \Z : \alpha(x) = ? \}$ the set of positions of ``?.'' We say $\alpha$ is a {\it partial word with undetermined part} $U$.

A partial word over the alphabet $\{a,?\}$ is called a \emph{simple partial word}, where $a \in {\mathcal A}$.

\textbf{Composition of partial words.} We are particularly interested in the partial words with undetermined part of the form $n \Z + l$ with $n \geq 2$ and $0 \leq l < n$. Let $\alpha$ be a partial word with undetermined part $n \Z + l$ and let $\beta$ be a partial word. We define the composition of $\alpha$ and $\beta$ by
$$
\alpha \vartriangleleft \beta(x) = \begin{cases} \alpha(x), & \text{ if } x\not \in n\Z+l, \\ \beta((x-l)/n), & \text{ if } x \in n \Z + l.\end{cases}
$$
Roughly speaking, we map the word $\beta$ into the undetermined part of $\alpha$.

\medskip

Assume $n \geq 2$. For $\varpi \in {\mathcal A}^{n-1}$, we can define a partial word $\beta^{(\varpi,n,l)}$ as follows:
$$
\beta^{(\varpi,n,l)}(x+n) = \beta^{(\varpi,n,l)}(x)\  (\forall x \in \Z)\ \  \text{ and }\ \  \beta^{(\varpi,n,l)}([l+1,l+n]) = \varpi
?.
$$
Thus $\beta^{(\varpi,n,l)}$ is a periodic word with period $n$ and undetermined part $n\Z+l$.

Now let us consider the composition of a sequence of partial words $\beta^{(\varpi^{(i)},n_i,l_i)}$. For $m \geq 1$, define
$$
\beta_m := \beta^{(\varpi^{(1)},n_1,l_1)} \lhd \beta^{(\varpi^{(2)},n_2,l_2)} \cdots \lhd\beta^{(\varpi^{(m)},n_m,l_m)},
$$
%
where the operation $\lhd$ is associative. Let $n_0 := 1$. It is easy to see that $\beta_m$ is a periodic partial word over ${\mathcal A}\cup\{?\}$ with undetermined part
$$
\D_m := {n_0 \cdots n_m} \Z + \sum_{i=1}^{m} n_0 \cdots n_{i-1} l_{i}.
$$
Notice that $\D_m \supset \D_{m+1}$. So $\beta_m$ is a sequence of partial words with smaller and smaller undetermined parts. In particular,
$$
\beta_m(x) = \beta_m(y), \ \  \text{ if } x \equiv y \pmod {n_0 n_1 \cdots n_m}.
$$

Define $\D_\infty := \bigcap_{m=1}^{\infty} \D_m$. It is easy to see that $\D_\infty$ is either empty or contains only one element. Let us denote the
limit of $\beta_m$ by
$$
\beta_\infty := \lim_{m \to \infty} \beta_m.
$$
Clearly $\beta_\infty$ is a partial word over ${\mathcal A}\cup\{?\}$ with undetermined part $\D_\infty$. It is called a {\it normal Toeplitz word} over ${\mathcal A}$ if $\D_\infty = \emptyset$ and it is non periodic.

Define $\tilde \A := \{ a \in \A : a \prec \varpi^{(k)} \text{ infinitely often} \}$ to be the set of {\it recurrent letters}. If $\beta_\infty$ is
not normal, then for any $a \in \tilde \A$, we can define a new word $\beta_\infty^{(a)}$ by
$$
\beta_\infty^{(a)}(x)=
\begin{cases}
\beta_\infty(x) & x \not\in \D_\infty, \\
a & x \in \D_\infty.
\end{cases}
$$
It is called an {\it extended Toeplitz word} if it is non periodic. Both normal Toeplitz words and extended Toeplitz words are called {\it Toeplitz words}.

The sequence $\{ (\varpi^{(k)},n_k,l_k) \}_{ k\geq 1}$ is called a {\it coding} of the corresponding Toeplitz word. If for any $k > 0$, $\varpi^{(k)} = a_k^{n_k-1}$ for some $a_k \in \A$ and $a_k\ne a_{k+1},$ then the resulting word is called a {\it simple Toeplitz word.}

If $\beta$ is a Toeplitz word over an alphabet of two letters, and $\beta$ is pattern Sturmian, we say that $\beta$ is a pattern Sturmian Toeplitz word. The following result is shown in \cite{GKTX06} for pattern Sturmian Toeplitz words:

\begin{prop}\label{pattern-sturm-toeplitz}
If $\beta$ is a pattern Sturmian Toeplitz word, then there exist a partial word $\beta^{(\varpi,n,l)}$ and a simple Toeplitz word $\tilde \beta$ such that
$$
\beta = \beta^{(\varpi,n,l)} \lhd \tilde \beta.
$$
\end{prop}

We remark that here $n=1$ is allowed. In this case, $l=0$ and $\varpi$ is an empty word. Thus $\beta^{(\varpi,n,l)}=?^\Z$ and  $\beta = \tilde \beta$ is indeed a simple Toeplitz word. It is known that if $\beta$ is a simple Toeplitz word over an alphabet of two letters, then $\beta $ is pattern Sturmian (see \cite{KZ02}).

Given a Toeplitz word $\beta$ over $\A$, let $\Omega_\beta$ be the closure of the orbit $\{ T^n \beta : n \in \Z\}$ of $\beta$ under the left shift $T$. It is clear that $\Omega_\beta$ is closed and invariant under $T$. $(\Omega_\beta,T)$ is called the {\it Toeplitz subshift} generated by $\beta$. Moreover, if $\{n_k\}_{k>0}$ is bounded, then $(\Omega_\beta,T)$ is called a {\it bounded Toeplitz subshift}; if $\beta$ is a simple Toeplitz word, then $(\Omega_\beta,T)$ is called a {\it simple Toeplitz subshift}. If $\{n_k\}_{k>0}$ is unbounded, then $(\Omega_\beta,T)$ is called an {\it  unbounded Toeplitz subshift}.

Let $\beta $ be a Toeplitz word with coding $\{(\varpi^{(k)},n_k,l_k)\}_{k \geq 1}$. We can assume $n_k \geq 3$ for any $k \in \N$. In fact if
$n_k = 2$ for some $k \in \N$, then
$$
\beta^{(\varpi^{(k)},n_k,l_k)} \lhd \beta^{(\varpi^{(k+1)},n_{k+1},l_{k+1})} = \beta^{(\tilde \varpi,\tilde n, \tilde l)}
$$
with $\tilde \varpi = \varpi^{(k)}\varpi^{(k+1)}_1 \cdots \varpi^{(k)} \varpi^{(k+1)}_{n_{k+1}-1} \varpi^{(k)}$, $\tilde n = n_k n_{k+1}$ and $\tilde l = l_k + n_k l_{k+1}.$ Since $\lhd$ is associative,
$$
\{(\varpi^{(1)},n_1,l_1), \cdots (\varpi^{(k-1)},n_{k-1},l_{k-1}), (\tilde \varpi, \tilde n, \tilde l), (\varpi^{(j)}, n_j, l_j) : j \geq k+2 \}
$$
is also a coding of $\beta$, but now $\tilde n = n_k n_{k+1} \geq 4 > 3$. Thus from now on we always assume $n_k \geq 3$ for any $k \in \N$.

Assume $\beta$ and $\tilde \beta$ have coding $\{ (\varpi^{(k)}, n_k, l_k) \}_{k \geq 1}$ and $\{ (\varpi^{(k)}, n_k, \tilde l_k) \}_{k \geq 1}$, respectively. Define
$$
p_m := \sum_{i=1}^{m} n_0 \cdots n_{i-1} l_{i}, \quad \tilde p_m := \sum_{i=1}^{m} n_0 \cdots n_{i-1} \tilde l_{i}.
$$
Then,
\begin{equation}\label{translation}
\beta(x) = \tilde \beta (x + (\tilde p_m - p_m)), \quad x \not\in \D_m.
\end{equation}
This can be shown easily by induction.

In \cite{JK} and \cite{KZ02}, only normal words over a $2$-letter alphabet are considered. However, for our purpose the extended words with $\D_\infty = \{ 0 \}$ have crucial importance.

Now we list two dynamical properties which we need later; for more details, see \cite{LQ11,LQ12}.

\begin{prop}\label{gene-toep}
For any Toeplitz word $\beta$, the system $(\Omega_\beta,T)$ is strictly ergodic.
\end{prop}

\begin{prop}\label{coding}
Let $\beta$ be a Toeplitz word with coding $\{ (\varpi^{(k)}, n_k, l_k) \}_{k \geq 1}$. Then any $\omega \in \Omega_\beta$ is a Toeplitz word with
coding $\{ (\varpi^{(k)}, n_k, \tilde{l}_k) \}_{k \geq 1}$ for some $\{ \tilde{l}_k \}_{k \ge 1} \subset \mathbb{N} \cup \{ 0 \}$.
\end{prop}

\subsection{Zero Measure Spectrum for Pattern Sturmian Toeplitz Subshifts}

Let $\beta$ be a pattern Sturmian Toeplitz word. By Proposition~\ref{pattern-sturm-toeplitz} there exist a partial word $\beta^{(\varpi,n,l)}$ and a simple Toeplitz word $\tilde \beta$ such that $\beta = \beta^{(\varpi,n,l)} \lhd \tilde \beta$. Assume that $\tilde \beta$ has coding $\{ (a_k, n_k, \tilde{l}_k) \}_{k \geq 1}$. Since $(\Omega_\beta,T)$ is strictly ergodic, the spectrum  $\sigma(H_\omega)$ is independent of $\omega \in \Omega_\beta$. Thus we can and will assume $l = 0$ and $\tilde l_k = 0$ for all $k \ge 1$.

\begin{lemma}
The subshift $(\Omega_\beta,T)$ satisfies condition (B).
\end{lemma}

\begin{proof}
The proof is the same as that of \cite[Proposition 4.1]{LQ11} since $\beta$ has the structure $\beta = \beta^{(\varpi,n,l)} \lhd \tilde \beta$ and $\beta$ is a word over the alphabet $\{ a, b \}$.
\end{proof}

\begin{coro}\label{c.zeromeasure}
There exists a compact set $\Sigma \subset \R$ such that $\Sigma = \sigma(H_{\beta^\prime})$ for any $\beta^\prime \in \Omega_\beta$ and $\Sigma$ is of Lebesgue measure zero. Moreover $M^E$ is uniform for all $E \in \R$.
\end{coro}

\begin{proof}
This is a consequence of \cite{DL06}.
\end{proof}

In the following we study the structure of the spectrum, which is needed when we discuss the point spectrum.

Recall that $\beta$ is a word over alphabet  $\{a,b\}$. Define $\bar{a}=b$ and $\bar{b}=a.$ Define $s_0 = \varpi a_1$ and $t_0 = \varpi \bar{a}_1$. Define inductively
$$
s_{k} = s_{k-1}^{n_{k}} a_{k}^{-1} a_{k+1},\ \ \  t_k = s_{k-1}^{n_{k}} \ \ (k \ge 1).
$$
In particular, $|s_k| = |t_k| =: \ell_k$ and we may observe the following:

\begin{prop}\label{trunc}
For every $k$, the words $s_k$ and $t_k$ are the same except for their respective rightmost symbol.
\end{prop}

An equivalent way to define these two sequences is
\begin{equation}\label{e.sktkrec}
s_{k} = s_{k-1}^{n_{k}-1} t_{k-1}, \ \ \ t_{k} = s_{k-1}^{n_{k}} \ \ (k\ge 1).
\end{equation}

Fix some $E \in \R$. We  define
$$
A_a := \begin{pmatrix} E - a & -1 \\ 1 & 0 \end{pmatrix}
\ \ \
A_b := \begin{pmatrix} E - b & -1 \\ 1 & 0 \end{pmatrix}.
$$
Define $A_{a_1 \cdots a_n} := A_{a_n} \cdots A_{a_1}$. Define $M_n := A_{s_n}$ and $h_n = {\rm tr} (M_n)$. Define $\sigma_n = \{ E \in \R : |h_n(E)| \le 2 \}$.

\begin{prop}\label{p.tracesandspectrum}
We have
$$
\Sigma = \sigma(H_\beta) = \bigcap_{n} \sigma_n \cup \sigma_{n+1}.
$$
\end{prop}

\begin{proof}
Define $D_k := A_{a_{k-1}^{-1} a_{k}}$. Then ${\rm tr}(D_k) = 2$ and $D_{k} D_{k-1} = I$ for any $k$. We have
\begin{eqnarray*}
M_{k+2} & = & D_{k+2} M_{k+1}^{n_{k+1}} \\
& = & S_{n_{k+1}}(h_{k+1}) D_{k+2} M_{k+1} - S_{n_{k+1}-1}(h_{k+1})D_{k+2} \\
& = & S_{n_{k+1}}(h_{k+1}) D_{k+2} D_{k+1} M_{k}^{n_k} -S_{n_{k+1}-1}(h_{k+1}) D_{k+2} \\
& = & S_{n_{k+1}}(h_{k+1}) M_{k}^{n_k} - S_{n_{k+1}-1}(h_{k+1}) D_{k+2} \\
& = & S_{n_{k+1}}(h_{k+1}) S_{n_k}(h_k) M_k - S_{n_{k+1}}(h_{k+1}) S_{n_k-1}(h_k) I \\
& & -S_{n_{k+1}-1}(h_{k+1})D_{k+2}.
\end{eqnarray*}
By taking the trace we get
$$
h_{k+2} = S_{n_{k+1}}(h_{k+1}) \huge( S_{n_k}(h_k) h_k - 2 S_{n_k-1}(h_k) \huge) - 2 S_{n_{k+1}-1}(h_{k+1}).
$$
By \cite[Lemma 3.1]{LQ11}, if $|h_k(E)|, |h_{k+1}(E)| > 2$, then $|h_{k+n}(E)| > 2$ for any $n \ge 2$. Thus for any $n \ge k+2$
$$
\sigma_{n} \subset \sigma_{k} \cup \sigma_{k+1}.
$$
On the other hand, we always have
$$
\sigma(H_\beta) \subset \bigcap_{k} \overline{\bigcup_{m \ge k} \sigma_m},
$$
Thus we get
$$
\sigma(H_\beta) \subset \bigcap_{k} \sigma_k\cup\sigma_{k+1}.
$$

Recall that $M^E$ is uniform for any $E \in \R$. Thus,
$$
\gamma(E,\omega) := \lim_{n \to\infty} \frac{\ln \|M^E_n(\omega)\|}{n} = \gamma(E)
$$
and the convergence is uniform for $\omega \in \Omega_\beta.$

Now assume $E \in \bigcap_{k} \sigma_k \cup \sigma_{k+1}.$ Similarly as in \cite{LQ11}, we can show that $\gamma(E, \beta) = 0$. Consequently $\gamma(E) = 0$ and $E \in \Sigma$.
\end{proof}

\subsection{Absence of Point Spectrum}

Now fix any $\omega \in \Omega_\beta$. By Proposition~\ref{coding}, $\omega = \beta^{(\varpi,n,\tilde l)} \lhd \hat \omega$ for some $\tilde l$ and some simple Toeplitz word $\hat \omega$ with coding $\{ (a_k, n_k, \hat{l}_k)\}_{k \geq 1}$.

\begin{prop}\label{partition}
$\omega$ has the following decomposition:
$$
\omega = \cdots t_k \ s_k^{\tau_{-2}}\ t_k \ s_k^{\tau_{-1}} \ t_k \ s_k^{\tau_{0}}\ t_k \ s_k^{\tau_{1}}\ t_k \ s_k^{\tau_{2}}\cdots
$$
with $\tau_j = n_{k}-1$ or $2n_{k}-1$ for all $j \in \Z$.

Moreover, this decomposition is unique. That is, for any other such decomposition, the $\tau_j$'s must be the same and the blocks of the decomposition must be aligned in the same way. In particular, the decomposition at level $k+1$ induces that at level $k$ via the rules \eqref{e.sktkrec}.
\end{prop}

\begin{proof}
By the definition of a Toeplitz word,
$$
\omega = \beta^{(\varpi,n,\tilde l)} \lhd \hat \omega = \beta^{(\varpi,n,\tilde l)} \lhd \beta^{(a_1,n_1,\hat l_1)} \lhd \cdots \lhd \beta^{(a_{k-1},n_{k-1},\hat l_{k-1})} \lhd \omega^{(k)},
$$
where $\omega^{(k)}$ is the simple Toeplitz word with coding $\{ (a_p, n_p, \hat{l}_p) \}_{p \geq k}$. It is seen that
$$
\beta^{(\varpi, n, \tilde l)} \lhd \beta^{(a_1, n_1, \hat l_1)} \lhd \cdots \lhd \beta^{(a_{k-1}, n_{k-1}, \hat l_{k-1})} =\beta^{(s_k^\ast, \ell_k,  l_k^\prime)},
$$
where $s_k^\ast = s_k a_k^{-1}$, $\ell_k = |s_k|$, and
$$
\omega^{(k)} = \cdots \ast(a_k^{n_k-1} a_{k+1})^{n_{k+1}-1} a_k^{n_k-1}\ast (a_k^{n_k-1} a_{k+1})^{n_{k+1}-1} a_k^{n_k-1}\ast \cdots,
$$
where $\ast$ is $a$ or $b$. Notice that $s_k = s_k^\ast a_k$ and $t_k = s_k^\ast a_{k+1}$. Thus,
\begin{eqnarray*}
\omega & = & \beta^{(s_k^\ast, \ell_k,  l_k^\prime)} \lhd \omega^{(k)} \\
& = & \cdots \underbrace{s_k^{n_k-1} t_k \cdots s_k^{n_k-1} t_k}_{n_{k+1}-1}s_k^{n_k-1} \Diamond \underbrace{s_k^{n_k-1} t_k\cdots s_k^{n_k-1} t_k}_{n_{k+1}-1}s_k^{n_k-1} \Diamond \cdots
\end{eqnarray*}
with $\Diamond \in \{s_k,t_k\}$. Then the result follows easily.

Next we show the uniqueness of the decomposition. Assume that
$$
\omega=\dots s_k^\ast \alpha_{-2}s_k^\ast \alpha_{-1}s_k^\ast \alpha_{0}s_k^\ast \alpha_{1}s_k^\ast \alpha_{2}\cdots =\dots s_k^\ast \theta_{-2}s_k^\ast \theta_{-1}s_k^\ast \theta_{0}s_k^\ast \theta_{1}s_k^\ast \theta_{2}\cdots
$$
are two decompositions of $\omega$ such that the occurrences of $\alpha_i$ are $\ell_k\Z+j$ and the occurrences of $\theta_i$ are $\ell_k\Z+\hat j$.  We only need to show that $j\equiv \hat j\pmod{\ell_k}.$  If otherwise, considering the second decomposition of $\omega$ relative to the first, we conclude that $\omega|_{\ell_k\Z+j}$ is a constant sequence. However by consider the first decomposition of $\omega$, we know that $\omega|_{\ell_k\Z+j}=\omega^{(k)}$ is a Toeplitz word, which is not periodic. Thus, we get a contradiction.

By this uniqueness,  the decomposition at level $k+1$ induces that at level $k$ via the rules \eqref{e.sktkrec}.
\end{proof}

Next, we recall Gordon-type criteria which establish a link between combinatorial properties of the sequences $\omega \in \Omega$ and non-decay properties of the solutions to
\begin{equation}\label{eve}
(H_\omega-E)\phi = 0.
\end{equation}
For proofs we refer the reader to \cite{DP86, S87}.

Fix some $\omega \in \Omega$ and some $E \in \R$. Let $\phi$ be a two-sided sequence that solves \eqref{eve} and obeys the normalization condition
\begin{equation}\label{normal}
|\phi(-1)|^2 + |\phi(0)|^2 = 1.
\end{equation}
Denote $\Phi(n) = (\phi(n),\phi(n-1))^T$. Then we have the following proposition.

\begin{prop}\label{gordon}
{\rm (a)} If for some $m \in \N$, we have $\omega_{-m+j} = \omega_j = \omega_{m+j}$, $0 \le j \le m-1$, then
$$
\max \left( \| \Phi(-m)\|, \| \Phi(m) \|, \| \Phi(2m) \| \right) \ge \tfrac{1}{2}.
$$
{\rm (b)} If for some $m \in \N$, we have that $\omega_0 \ldots \omega_{2m-1}$ is a cyclic permutation of $s_n s_n$, then
$$
\max \left( |h_n(E)| \cdot \| \Phi(m) \|, \| \Phi(2m) \| \right) \ge \tfrac{1}{2}.
$$
Analogous conclusions hold if the assumptions in {\rm (a)} and {\rm (b)} are reflected at the origin.
\end{prop}

We see that we obtain useful estimates for the solutions $\phi$ of \eqref{eve} if we exhibit appropriate squares and cubes in $\omega$.

\begin{prop}\label{p.contspectrum}
For every $\omega \in \Omega$, the point spectrum of $H_\omega$ is empty.
\end{prop}

\begin{proof}
Fix $\omega \in \Omega$, $E \in \Sigma$, and a solution $\phi$ to \eqref{eve} obeying \eqref{normal}. We want to prove that $\phi$ is not square-summable. We shall show that given any $n \in \N$, there exists $m \in \Z$ with $|m| \ge n$ such that $\| \Phi(m) \| \ge \frac{1}{4}$. From this the assertion clearly follows.

So let $n \in \N$ be fixed and pick $k \in \N$ such that $\ell_k = |s_k| = |t_k| \ge n$. Consider the $k$-partition of $\omega$.

\textit{Case 1: The site $0 \in \Z$ is contained in an $s_k$-block and this $s_k$-block is followed to the right by an $s_k$-block. In other words, we have $\hat s_k s_k$.} (Here and in the following, the hat-symbol marks the block that contains the site $0 \in \Z$.) Because of Proposition \ref{partition} there are two subcases.

\textit{Case 1.1: We have $s_k \hat{s}_k s_k$.} Because of Propositions~\ref{trunc} and \ref{partition} we can conclude by applying Proposition~\ref{gordon}~(a) with $m = \ell_k$.

\textit{Case 1.2: We have $t_k \hat{s}_k s_k$.} Then, either $|h_k(E)| \le 2$ and we are done in this case by Proposition~\ref{gordon}~(b), or $|h_k(E)| > 2$ and then $|h_{k+1}(E)| \le 2$ by Proposition~\ref{p.tracesandspectrum}. In the latter case we consider the $(k+1)$-partition where we must have either $s_{k+1} \hat{t}_{k+1}$ or $s_{k+1} \hat s_{k+1}$. Note that the origin is not the rightmost site in the block with the hat.

\textit{Case 1.2.1: We have $s_{k+1} \hat{t}_{k+1}$.} Then we have either $s_{k+1} s_{k+1} \hat{t}_{k+1}$ or $t_{k+1} s_{k+1} \hat t_{k+1}$.

\textit{Case 1.2.1.1: We have $s_{k+1} s_{k+1} \hat t_{k+1}$.} In this case we can conclude immediately by applying Proposition~\ref{gordon}~(b) (reflected at the origin) because we have $|h_{k+1}(E)| \le 2$ and the origin is not the rightmost site in $\hat t_{k+1}$.

\textit{Case 1.2.1.2: We have $t_{k+1} s_{k+1} \hat t_{k+1}$.} In this case we must have $s_{k+2} \hat s_{k+2}$ and the origin is not the rightmost site in $\hat s_{k+2}$. By Proposition~\ref{partition} we have $s_{k+2} s_{k+2} \hat s_{k+2}$ or $s_{k+2} \hat s_{k+2} s_{k+2}$.

\textit{Case 1.2.1.2.1: We have $s_{k+2} s_{k+2} \hat s_{k+2}$.} In this case we can conclude immediately by applying Proposition~\ref{gordon}~(a) (reflected at the origin) due to Propositions~\ref{trunc} and \ref{partition} and the fact that the origin is not the rightmost site in $\hat s_{k+2}$.

\textit{Case 1.2.1.2.2: We have $s_{k+2} \hat s_{k+2} s_{k+2}$.} Similarly to Case 1.1, we can conclude immediately by applying Proposition~\ref{gordon}~(a) due to Propositions~\ref{trunc} and \ref{partition} and the fact that the origin is not the rightmost site in $\hat s_{k+2}$.

\textit{Case 1.2.2: We have $s_{k+1} \hat s_{k+1}$.} This follows by the same argument as the one used in Case 1.2.1.2, with all indices reduced by one. This closes Case 1.

\textit{Case 2: In the $k$-partition we have $s_k \hat{s}_k$.} This case can be treated analogously to Case 1.2.1.2, with all indices reduced by two. This closes Case 2.

\textit{Case 3: In the $k$-partition we have $t_k \hat{s}_k t_k$.} We pass to the $(k+1)$-partition where we must have $s_{k+1} \hat{s}_{k+1}$ due to the uniqueness statement in Proposition~\ref{partition} and we can then proceed analogously to Case 1.2.2. This closes Case 3.

\textit{Case 4: In the $k$-partition we have $\hat{t}_k$.} Thus, due to the uniqueness statement in Proposition~\ref{partition}, in the $(k+1)$-partition we have $\hat{s}_{k+1}$ and we are in one of the Cases 1--3 with all indices increased by one. In particular, we obtain a sufficient solution estimate. This closes Case 4 and hence concludes the proof.
\end{proof}

Note that we have established all the statements in Theorem~\ref{t.patternsturmian}:

\begin{proof}[Proof of Theorem~\ref{t.patternsturmian}.]
Suppose $V$ is a pattern Sturmian Toeplitz potential and $H$ is the associated Schr\"odinger operator. By Corollary~\ref{c.zeromeasure}, $H$ has zero measure spectrum, which in turn also implies that $H$ has empty absolutely continuous spectrum. On the other hand, by Proposition~\ref{p.contspectrum}, $H$ has empty point spectrum.
\end{proof}

\section{Sparse Pattern Sturmian Sequences}\label{s.sparse}

In this section we discuss sparse pattern Sturmian sequences, the third class mentioned in Example~\ref{example}. Recall that these one-sided sequences take the form
\begin{equation}\label{e.sparsepss}
V(n) = \begin{cases} v & n = n_k \text{ for some } k, \\ 0 & \text{otherwise}, \end{cases}
\end{equation}
where $v \in \R \setminus \{ 0 \}$ and
\begin{equation}\label{e.sparsepss2}
n_{k+1} > 2 n_k \ge 1.
\end{equation}
Note that $n_k$ grows at least as fast as $2^k$ and hence the gaps $n_{k+1} - n_k > n_k$ grow at least geometrically. Of course they may grow faster, and this case will be of interest below. But even in the case of geometric growth of gaps, we will be able to prove several statements in the desired direction.

We have
$$
\sigma_\mathrm{ess}(H_\phi) = [-2,2] \cup \{ \mathrm{sgn} \, v \cdot \sqrt{4 + v^2} \}
$$
by \cite[Theorem~3.13]{CFKS}. In particular the spectrum outside $[-2,2]$ is non-empty and pure point. We claim that it is purely singular continuous inside $(-2,2)$ under suitable assumptions. This will follow from the next two lemmas.

\begin{lemma}
Suppose $V$ is of the form \eqref{e.sparsepss}, subject to \eqref{e.sparsepss2}. Then, for every boundary condition $\phi$, the absolutely continuous spectrum of the discrete Schr\"odinger operator $H_\phi$ in $\ell^2(\Z_+)$ is empty.
\end{lemma}

\begin{proof}
This follows from \cite[Theorem~1.1]{R11}.
\end{proof}

\begin{lemma}
Suppose $V$ is of the form \eqref{e.sparsepss}, subject to \eqref{e.sparsepss2}, and $E \in (-2,2)$. Let
$$
C_E := \sup \left\{ \left\| \begin{pmatrix} E & -1 \\ 1 & 0 \end{pmatrix}^j \right\| : j \in \Z_+ \right\} < \infty
$$
and
$$
O_{E,v} := \left\| \begin{pmatrix} E - v & -1 \\ 1 & 0 \end{pmatrix} \right\|
$$
If
$$
\sum_{k = 1}^\infty \frac{n_{k+1} - n_k}{C_E^{2k} O_{E,v}^{2k}} = \infty,
$$
then $E$ is not an eigenvalue of $H_\phi$ for any $\phi$.

In particular, if the gaps $n_{k+1} - n_k$ grow super-geometrically {\rm (}resp., at least geometrically of a high enough rate {\rm (}depending on $v$ and $\delta > 0${\rm ))}, then the point spectrum of $H_\phi$ in $(-2,2)$ {\rm (}resp., in $(-2+\delta,2-\delta)${\rm )} is empty for every boundary condition $\phi$.
\end{lemma}

\begin{proof}
This follows via a modification of the proof of \cite[Theorem~4.1]{SS96}.
\end{proof}

Based on work of Zlato\v{s} \cite{Z04}, one can take the analysis of these sparse potentials further and establish upper and lower bounds for the local scaling exponents of the singular continuous spectral measures that arise in this context under suitable assumptions on the growth of the gaps $n_{k+1} - n_k$ and the size $v$ of the barriers.

It is a natural question whether there are two-sided extensions of these sparse pattern Sturmian sequences that still have maximal pattern complexity $p^*(n) = 2n$. This is of interest as the bulk of the paper is concerned with Schr\"odinger operators on the whole line, and the restriction to a half-line usually generates additional discrete spectrum which is a pure half-line phenomenon. (Specifically, for uniformly recurrent potentials, there is no discrete spectrum in the two-sided case, while there typically is some discrete spectrum in the one-sided case.)

Notice that the sparse examples above are not only non-recurrent, but rather strongly non-recurrent. Here we say that a one-sided sequence is \emph{eventually recurrent} if after removal of a suitable finite prefix the sequence becomes recurrent. Any one-sided sequence that is not eventually recurrent is called \emph{strongly non-recurrent}.

Let us show that it is impossible for a two-sided extension of the sparse pattern Sturmian sequences discussed above to have maximal pattern complexity $p^*(n) = 2n$. More generally, we have the following statement.

\begin{prop}\label{p.nonrec2sided}
Suppose $u \in \mathcal{A}^{\Z_+}$ is a strongly non-recurrent one-sided pattern Sturmian sequence. Then for any two-sided extension $\tilde u$ of $u$ {\rm (}i.e., $\tilde u \in \mathcal{A}^\Z$ with $\tilde u|_{\Z_+} = u${\rm )}, the maximal pattern complexity function of $\tilde u$ obeys $p^*(n) > 2n$ for infinitely many $n$.
\end{prop}

The following result is crucial for our proof of Proposition~\ref{p.nonrec2sided}.

\begin{prop}\label{p.quasisturmian}
Suppose that $u \in \mathcal{A}^{\Z_+}$ is strongly non-recurrent. Then, for every $j \in \Z_+$, we have $p_u(n) > n + j$ for infinitely many $n$.
\end{prop}

\begin{proof}
Suppose to the contrary that there exists $j \in \Z_+$ such that for $n \ge n_0$, we have $p_u(n) \le n+j$. This implies that $u$ is quasi-Sturmian, that is, it is of the form $u = wS(\tilde u)$, where $w$ is a finite word and $\tilde u$ is a Sturmian sequence; compare \cite[Proposition~8]{C98}. Since every Sturmian sequence is known to be recurrent, this implies that $u$ is eventually recurrent.
\end{proof}

\begin{proof}[Proof of Proposition~\ref{p.nonrec2sided}.]
Suppose $u \in \mathcal{A}^{\Z_+}$ is a strongly non-recurrent one-sided pattern Sturmian sequence and $\tilde u$ is a two-sided extension of $u$. We will show that the block complexity function of $\tilde u$ obeys $p_{\tilde u}(n) > 2n$ for infinitely many $n$. Since the block complexity function bounds the maximal pattern complexity function from below, this implies the assertion. In fact, we will even prove that for every $j$, we have $p_{\tilde u}(n) > 2n + j$ for infinitely many $n$.

Since $u$ is strongly non-recurrent (here, mere non-recurrence suffices), there exist a finite word $v \in \mathcal{A}^*$ and a $k \in \Z_+$ such that
\begin{itemize}

\item[(i)] $v$ occurs in $u$ at position $k$, that is, $u_k \ldots u_{k+|v|-1} = v$;

\item[(ii)] there is no occurrence of $v$ in $u$ to the right of position $k$, that is, for every $\tilde k > k$, we have  $u_{\tilde k} \ldots u_{\tilde k+|v|-1} \not= v$.

\end{itemize}

By Proposition~\ref{p.quasisturmian}, for any $j$, there are infinitely many $n$ with $p_u(n) > n + k + |v| + j$. Let us consider those $n$ from this infinite set for which we have in addition $n > k + |v|$.

Given such a value of $n$, consider the two-sided extension $\tilde u$ of $u$ in question. Notice that the words
$$
w^{(m)} = \tilde u_{k + |v| - n + m} \tilde u_{k + |v| - n + m + 1} \cdots \tilde u_{k + |v| - n + m + (n-1)}, \; m = 1, \ldots, n - (k + |v|)
$$
have the following properties:
\begin{itemize}

\item[(a)] each $w^{(m)}$ has length $n$;

\item[(b)] each $w^{(m)}$ contains a copy of $v$;

\item[(c)] the words $w^{(m)}$ are distinct and hence there are $n - (k + |v|)$ many of them;

\item[(d)] none of the words $w^{(m)}$ occurs in $u$ and hence is not counted among the $p_u(n)$ words of length $n$ that do occur in $u$.

\end{itemize}

Statement (a) is obvious, statements (b) follows from (i), and statements (c) and (d) follow from (b) and (ii). By statements (a), (c), and (d) it follows that
$$
p_{\tilde u}(n) > \big[ n + k + |v| + j \big] + \big[ n - (k + |v|) \big] = 2n + j.
$$

Since the block complexity function bounds the maximal pattern complexity function from below, it follows that we have $p^*_{\tilde u}(n) > 2n + j$ for infinitely many values of $n$ as well. This concludes the proof.
 \end{proof}

\end{document}